\documentclass[12pt]{amsart}
\usepackage{a4wide}
\usepackage[T1]{fontenc}
\usepackage[utf8x]{inputenc}
\usepackage[english]{babel}
\usepackage{amscd,amsmath,amsthm,amssymb,graphics}
\usepackage{lmodern,pst-node}
\usepackage{multicol}
\usepackage{amsfonts,amssymb,amscd,amsmath,enumitem,verbatim}
\usepackage[pdfencoding=auto,bookmarks=true,bookmarksnumbered=true]{hyperref}

\def\frk{\frak}               

\def\Phi{{\frk n}}
\def\Phi{{\frk N}}
\def\opn#1#2{\def#1{\operatorname{#2}}} 
\opn\chara{char} \opn\length{\ell} \opn\pd{pd} \opn\rk{rk}
\opn\projdim{proj\,dim} \opn\injdim{inj\,dim} \opn\rank{rank}
\opn\depth{depth} \opn\grade{grade} \opn\height{height}
\opn\embdim{emb\,dim} \opn\codim{codim}

\opn\Tr{Tr} \opn\bigrank{big\,rank}
\opn\superheight{superheight}\opn\lcm{lcm}
\opn\trdeg{tr\,deg}
\opn\reg{reg} \opn\lreg{lreg} \opn\ini{in} \opn\lpd{lpd}
\opn\size{size}\opn\bigsize{bigsize}
\opn\cosize{cosize}\opn\bigcosize{bigcosize}
\opn\sdepth{sdepth}\opn\sreg{sreg}
\opn\link{link}\opn\fdepth{fdepth}
\opn\div{div} \opn\Div{Div} \opn\cl{cl} \opn\Cl{Cl}
\opn\Spec{Spec} \opn\Supp{Supp} \opn\supp{supp} \opn\Sing{Sing}
\opn\Ass{Ass} \opn\Min{Min}\opn\Mon{Mon} \opn\dstab{dstab} \opn\astab{astab}
\opn\Syz{Syz}
\opn\Ann{Ann} \opn\Rad{Rad} \opn\Soc{Soc}
\opn\Im{Im} \opn\Ker{Ker} \opn\Coker{Coker} \opn\Am{Am}
\opn\Hom{Hom} \opn\Tor{Tor} \opn\Ext{Ext} \opn\End{End}
\opn\Aut{Aut} \opn\id{id}

\opn\nat{nat}
\opn\pff{pf}
\opn\Pf{Pf} \opn\GL{GL} \opn\SL{SL} \opn\mod{mod} \opn\ord{ord}
\opn\Gin{Gin} \opn\Hilb{Hilb}\opn\sort{sort}
\opn\initial{init}
\opn\ende{end}
\opn\height{height}
\opn\type{type}
\opn\aff{aff} \opn\con{conv} \opn\relint{relint} \opn\st{st}
\opn\lk{lk} \opn\cn{cn} \opn\core{core} \opn\vol{vol}
\opn\link{link} \opn\star{star}\opn\lex{lex}
\opn\gr{gr}

\def\pot#1#2{#1[\kern-0.28ex[#2]\kern-0.28ex]}

\opn\dirlim{\underrightarrow{\lim}}
\opn\inivlim{\underleftarrow{\lim}}
\def\Implies{\ifmmode\Longrightarrow \else
        \unskip${}\Longrightarrow{}$\ignorespaces\fi}
\def\implies{\ifmmode\Rightarrow \else
        \unskip${}\Rightarrow{}$\ignorespaces\fi}
\def\iff{\ifmmode\Longleftrightarrow \else
        \unskip${}\Longleftrightarrow{}$\ignorespaces\fi}

\let\:=\colon
\makeatletter
\CheckCommand*\refstepcounter[1]{\stepcounter{#1}%
\protected@edef\@currentlabel
{\csname p@#1\endcsname\csname the#1\endcsname}%
}
\renewcommand*\refstepcounter[1]{\stepcounter{#1}%
\protected@edef\@currentlabel
{\csname p@#1\expandafter\endcsname\csname the#1\endcsname}%
}
\def\labelformat#1{\expandafter\def\csname p@#1\endcsname##1}
\DeclareRobustCommand\Ref[1]{\protected@edef\@tempa{\ref{#1}}%
\expandafter\MakeUppercase\@tempa
}
\makeatother
										
\makeatletter
\newcommand{\numberlike}[2]{%
\expandafter\def\csname c@#1\endcsname{%
\expandafter\csname c@#2\endcsname}%
}
\makeatother

\def\DefaultNumberTheoremWithin{section}

\theoremstyle{plain}
\newtheorem{Lemma}{Lemma}
\numberwithin{Lemma}{\DefaultNumberTheoremWithin}
\labelformat{Lemma}{Lemma~#1}

\numberwithin{Claim}{\DefaultNumberTheoremWithin}
\numberlike{Claim}{Lemma}
\labelformat{Claim}{Claim~#1}
\newtheorem{Theorem}{Theorem}
\numberwithin{Theorem}{\DefaultNumberTheoremWithin}
\numberlike{Theorem}{Lemma}
\labelformat{Theorem}{Theorem~#1}
\newtheorem{Corollary}{Corollary}
\numberwithin{Corollary}{\DefaultNumberTheoremWithin}
\numberlike{Corollary}{Lemma}
\labelformat{Corollary}{Corollary~#1}
\newtheorem{Proposition}{Proposition}
\numberwithin{Proposition}{\DefaultNumberTheoremWithin}
\numberlike{Proposition}{Lemma}
\labelformat{Proposition}{Proposition~#1}

\numberwithin{Conjecture}{\DefaultNumberTheoremWithin}
\numberlike{Conjecture}{Lemma}
\labelformat{Conjecture}{Conjecture~#1}
									
\theoremstyle{definition}

\numberwithin{Definition}{\DefaultNumberTheoremWithin}
\numberlike{Definition}{Lemma}
\labelformat{Definition}{Definition~#1}
														
\theoremstyle{definition}

\numberwithin{Question}{\DefaultNumberTheoremWithin}
\numberlike{Question}{Lemma}
\labelformat{Question}{Question~#1}
																			
\theoremstyle{definition}

\numberwithin{Problem}{\DefaultNumberTheoremWithin}
\numberlike{Problem}{Lemma}
\labelformat{Problem}{Problem~#1}

\theoremstyle{remark}

\numberwithin{Remark}{\DefaultNumberTheoremWithin}
\numberlike{Remark}{Lemma}
\labelformat{Remark}{Remark~#1}
\theoremstyle{remark}

\numberwithin{Example}{\DefaultNumberTheoremWithin}
\numberlike{Example}{Lemma}

\let\epsilon\varepsilon
\let\kappa=\varkappa
\def\qed{\ifhmode\textqed\fi
      \ifmmode\ifinner\quad\qedsymbol\else\dispqed\fi\fi}
\def\textqed{\unskip\nobreak\penalty50
       \hskip2em\hbox{}\nobreak\hfil\qedsymbol
       \parfillskip=0pt \finalhyphendemerits=0}
\def\dispqed{\rlap{\qquad\qedsymbol}}

\opn\dis{dis}
\def\pnt{{\raise0.5mm\hbox{\large\bf.}}}

\opn\Lex{Lex}

\begin{document}
\title{The Castelnuovo-Mumford regularity of binomial edge ideals}

\author{ Dariush Kiani and Sara Saeedi Madani }


\email{dkiani@aut.ac.ir, dkiani7@gmail.com}
\email{sarasaeedim@gmail.com}

\begin{abstract}
We prove a conjectured upper bound for the Castelnuovo-Mumford regularity of binomial edge ideals of graphs, due to Matsuda and Murai. Indeed, we prove that $\mathrm{reg}(J_G)\leq n-1$ for any graph $G$ with $n$ vertices, which is not a path. Moreover, we study the behavior of the regularity of binomial edge ideals under the join product of graphs.
\end{abstract}


\maketitle

\section{ Introduction }\label{Introduction}

\noindent The binomial edge ideal of a graph was introduced in \cite{HHHKR}, and \cite{O} at about the same time.
Let $G$ be a finite simple graph with vertex set $[n]$ and edge set $E(G)$.
Also, let $S=K[x_1,\ldots,x_n,y_1,\ldots,y_n]$ be the polynomial ring over a field $K$. Then the \textbf{binomial edge ideal} of $G$ in $S$,
denoted by $J_G$, is generated by binomials $f_{ij}=x_iy_j-x_jy_i$, where $i<j$ and $\{i,j\}\in E(G)$. Also, one could see this ideal as an ideal
generated by a collection of 2-minors of a $(2\times n)$-matrix whose entries are all indeterminates. Many of the algebraic properties and invariants of such ideals
were studied in \cite{D,EHH,EZ,HHHKR,KS,KS1,MM}, \cite{SK,SK1,SZ} and \cite{ZZ}. One of these invariants is the Castelnuovo-Mumford regularity. Recall that the Castelnuovo-Mumford regularity (or simply, regularity) of a graded $S$-module $M$ is defined as
\[
\mathrm{reg}(M)=\mathrm{max}\{j-i~:~\beta_{i,j}(M)\neq 0\}.
\]

In \cite{SK1}, the authors posed a conjecture about the Castelnuovo-Mumford regularity of the binomial edge ideal of graphs in terms of the number of maximal cliques of $G$, which is denoted by $c(G)$. \\

\noindent \textbf{Conjecture~A.} Let $G$ be a graph. Then $\mathrm{reg}(J_G)\leq c(G)+1$.\\

In \cite{SK}, the authors proved this conjecture for closed graphs, i.e. the graphs whose binomial edge ideals have a quadratic Gr\"{o}bner basis with respect to the lexicographic order induced by $x_1>\cdots >x_n>y_1>\cdots >y_n$.

In \cite{MM}, Matsuda and Murai gained an upper bound for the regularity of the binomial edge ideal of a graph on $n$ vertices as follows.

\begin{Theorem}\label{Matsuda-Murai}
\cite[Theorem~1.1]{MM} Let $G$ be a graph on $n$ vertices. Then $\mathrm{reg}(J_G)\leq n$.
\end{Theorem}

This theorem, in particular, shows that Conjecture~A is valid for trees. Ene and Zarojanu in \cite{EZ} proved Conjecture~A for block graphs which are included in the class of chordal graphs.

In \cite[Corollary~2.3]{MM} (see also \cite[Corollary~17]{SK1}), the authors gave a lower bound for the binomial edge ideal of a graph. Indeed, if $l$ is the length of the longest induced path of the graph $G$, then $\reg(J_G)\geq l+1$.

Matsuda and Murai in \cite{MM} also posed a conjecture which says that $\mathrm{reg}(J_G)=n$ if and only if $G=P_n$, the path over $n$ vertices. In other words, by support of \ref{Matsuda-Murai}, one can reformulate this conjecture as the following. \\

\noindent \textbf{Conjecture~B.} Let $G$ be a graph on $n$ vertices which is not a path. Then $\mathrm{reg}(J_G)\leq n-1$. \\

Ene and Zarojanu in \cite{EZ} showed that Conjecture~B holds for block graphs. Moreover, Conjecture~B was proved for cycles by Zahid and Zafar in \cite{ZZ}.

In this paper, we study the regularity of the binomial edge ideals and especially Matsuda and Murai's conjecture. This paper is organized as follows.

In Section~\ref{Preliminaries}, we pose some definitions, facts and notation which will be used throughout the paper.

In Section~\ref{Murai}, we prove the conjecture of Matsuda and Murai (Conjecture B) which is the main result of this paper.

In Section~\ref{join}, we show that the regularity of the binomial edge ideal of the join product of two graphs $G_1$ and $G_2$, which are not both complete, is equal to $\mathrm{max}\{\mathrm{reg}(J_{G_1}),\mathrm{reg}(J_{G_2}),3\}$. Applying this fact, then we generalize some results of Schenzel and Zafar about complete $t$-partite graphs. Moreover, our result on the regularity of the join product of two graphs shows that if $\mathcal{A}$ is the set of all graphs for which the Conjecture A holds, then $\mathcal{A}$ is closed under the join product.

Throughout the paper, we mean by a graph $G$, a simple graph. Moreover, if $V=\{v_1,\ldots,v_n\}$ is the vertex set of $G$ (which contains $n$ elements), then for simplicity we denote it by $[n]$.

\section{ Preliminaries }\label{Preliminaries}

\noindent In this section, we review some notions and facts around graphs and binomial edge ideals, which we need throughout.

A vertex $v$ of the graph $G$ for which the induced subgraph of $G$ on $N_G(v)$ is a complete graph, is called a \textbf{simplicial vertex}. By $N_G(v)$, we mean the set of all neighbors (i.e. adjacent vertices) of the vertex $v$ in $G$.

A vertex $v$ of the graph $G$ whose deletion from the graph, implies a graph with more connected components than $G$, is called a \textbf{cut point} of $G$.

Let $G$ be a graph and $e=\{v,w\}$ an edge of it. If $\{e_1,\ldots,e_t\}$ is a set of edges of $G$, then by $G\setminus \{e_1,\ldots,e_t\}$, we mean the graph on the same vertex set as $G$ in which the edges $e_1,\ldots,e_t$ are omitted. Here, we simply write $G\setminus e$, instead of $G\setminus \{e\}$.

Let $G=(V,E)$ be a graph and $v,w$ be two vertices of $G$, and assume that $e=\{v,w\}$ is not an edge of $G$. Then $G\cup e$ is the graph on the same vertex set as $G$ and the edge set $E\cup \{e\}$. Moreover, as it was used in \cite{MSh}, $G_e$ is defined to be the graph on the vertex set $V$, and the edge set $E \cup \{\{x,y\}~:~x,y\in N_G(v)~\mathrm{or}~x,y\in N_G(w)\}$. \\

Let $G$ and $H$ be two graphs on $[m]$ and $[n]$, respectively. We denote by $G*H$, the \textbf{join product} of two graphs $G$ and $H$, that is
the graph with vertex set $[m]\cup [n]$, and the edge set $E(G)\cup E(H)\cup \{\{v,w\}~:~v\in [m],~w\in [n]\}$. Let $V$ be a set. To simplify our notation throughout this paper, we introduce the \textbf{join} of two collection of subsets of $V$, $\mathcal{A}$ and $\mathcal{B}$, denoted by $\mathcal{A}\circ \mathcal{B}$, as $\{A\cup B: A\in \mathcal{A}, B\in \mathcal{B}\}$. If $\mathcal{A}_1,\ldots,\mathcal{A}_t$ are collections of subsets of $V$, then we denote their join, by $\bigcirc_{i=1}^{t}\mathcal{A}_i$. \\

Suppose that $G$ is a graph on $[n]$. Let $T$ be a subset of $[n]$, and let $G_1,\ldots,G_{c_G(T)}$ be the connected
components of $G_{[n]\setminus T}$, the induced subgraph of $G$ on $[n]\setminus T$. For each $G_i$, we denote by $\widetilde{G}_i$ the complete graph on the vertex set $V(G_i)$. If there is no confusion, then we may simply write $c(T)$ instead of $c_G(T)$. Set $$P_T(G)=(\bigcup_{i\in T}\{x_i,y_i\}, J_{\widetilde{G}_1},\ldots,J_{\widetilde{G}_{c(T)}}).$$ Then, $P_T(G)$ is a prime ideal, where $\mathrm{height}\hspace{0.35mm}P_T(G)=n+|T|-c(T)$, by \cite[Lemma~3.1]{HHHKR}. Moreover, $J_G=\bigcap_{T\subset [n]}P_T(G)$, by \cite[Theorem~3.2]{HHHKR}. So that, $\mathrm{dim}\hspace{0.35mm}S/J_G=\mathrm{max}\{n-|T|+c(T):T\subset [n]\}$, by \cite[Cororally~3.3]{HHHKR}. If each $i\in T$ is a cut point of the graph $G_{([n]\setminus T)\cup \{i\}}$, then we say that $T$ has \textbf{cut point property} for $G$. Let $\mathcal{C}(G)=\{\emptyset\}\cup \{T\subset [n]:T~\mathrm{has~cut~point~property~for}~G\}$. One has $\mathcal{C}(G)=\{\emptyset\}$ if and only if $G$ is a complete graph. Denoted by $\mathcal{M}(G)$, we mean the set of all minimal prime ideals of $J_G$. Then, one has $T\in \mathcal{C}(G)$ if and only if $P_T(G)\in \mathcal{M}(G)$, by \cite[Corollary~3.9]{HHHKR}.

\section{ Matsuda and Murai's conjecture }\label{Murai}

\noindent In this section, we prove the conjecture of Matsuda and Murai on the Castelnuovo-Mumford regularity of the binomial edge ideal of a graph. In the sequel, we use the following proposition.

\begin{Proposition}\label{reg-chordal}
Let $H$ be a graph and $e$ be an edge of $H$. Then we have
\begin{itemize}
\item[\em (a)] $\mathrm{reg}(J_H)\leq \mathrm{max}\{\mathrm{reg}(J_{H\setminus e}), \reg(J_{H\setminus e}:f_e)+1\}$;
\item[\em (b)] $\mathrm{reg}(J_{H\setminus e})\leq \mathrm{max}\{\mathrm{reg}(J_H), \reg(J_{H\setminus e}:f_e)+2\}$;
\item[\em (c)] $\reg(J_{H\setminus e}:f_e)+2\leq \mathrm{max}\{\mathrm{reg}(J_{H\setminus e}), \mathrm{reg}(J_{H})+1\}$.
\end{itemize}
\end{Proposition}

\begin{proof}
It suffices to consider the short exact sequence $$0\longrightarrow S/(J_{H\setminus e}:f_e)(-2)\stackrel{f_e} \longrightarrow S/J_{H\setminus e}\longrightarrow S/J_H\rightarrow 0.$$ Then, the statement follows by \cite[Corollary~18.7]{P}.
\end{proof}

We also benefit from the following theorem which appeared in \cite[Theorem~3.7]{MSh}, and gives a system of generators of the ideal $J_{G\setminus e}:f_e$, explicitly.

\begin{Theorem}\label{M-colon2}
\cite[Theorem~3.7]{MSh} Let $G$ be a graph and $e=\{i,j\}$ be an edge of $G$. Then we have
$$J_{G\setminus e}:f_e=J_{{(G\setminus e)}_e}+I_G,$$
where $I_G=(g_{P,t}~:~P:i,i_1,\ldots,i_s,j~\mathrm{is~a~path~between~}i,j~\mathrm{in}~G~\mathrm{and}~0\leq t\leq s)$, $g_{P,0}=x_{i_1}\cdots x_{i_s}$ and for every $1\leq t\leq s$, $g_{P,t}=y_{i_1}\cdots y_{i_t}x_{i_{t+1}}\cdots x_{i_s}$.
\end{Theorem}

The following lemma is also needed to prove the main theorem.

\begin{Lemma}\label{colon3}
Let $G$ be a graph on $[n]$, $v$ a simplicial vertex of $G$ with $\mathrm{deg}_{G}(v)\geq 2$, and $e$ an edge incident with $v$. Then we have $\reg(J_{G\setminus e}:f_e)\leq n-2$.
\end{Lemma}

\begin{proof}
Let $v_1,\ldots,v_t$ be all the neighbors of the simplicial vertex $v$, and $e_1,\ldots,e_t$ be the edges joining $v$ to $v_1,\ldots,v_t$, respectively, where $t\geq 2$. Without loss of generality, assume that $e:=e_t$. Note that for each $i=1,\ldots,t-1$, $v,v_i,v_t$ is a path between $v$ and $v_t$ in $G$, so that for all $i=1,\ldots,t-1$, $x_i$ and $y_i$ are in the minimal monomial set of generators of $I_G$. Also, all other paths between $v$ and $v_t$ in $G\setminus e$ contain $v_i$ for some $i=1,\ldots,t-1$. Thus, all the monomials corresponding to these paths, are divisible by either $x_i$ or $y_i$ for some $i=1,\ldots,t-1$. Hence, we have
$I_G=(x_i,y_i:1\leq i\leq t-1)$. So that $J_{G\setminus e}:f_e=J_{{(G\setminus e)}_e}+(x_i,y_i:1\leq i\leq t-1)$. The binomial generators of $J_{{(G\setminus e)}_e}$ corresponding to the edges containing vertices $v_1,\ldots,v_{t-1}$, are contained in $I_G$. Let $H:={(G\setminus e)}_e$. Then, we have $J_{G\setminus e}:f_e=J_{H_{[n]\setminus \{v,v_1,\ldots,v_{t-1}\}}}+(x_i,y_i:1\leq i\leq t-1)$, since $v$ is an isolated vertex of $H_{[n]\setminus \{v_1,\ldots,v_{t-1}\}}$. Thus, $\reg(J_{G\setminus e}:f_e)=\reg(J_{H_{[n]\setminus \{v,v_1,\ldots,v_{t-1}\}}})$. But, $\reg(J_{H_{[n]\setminus \{v,v_1,\ldots,v_{t-1}\}}})\leq n-2$, by \ref{Matsuda-Murai}, since $t\geq 2$. Therefore, $\reg(J_{G\setminus e}:f_e)\leq n-2$, as desired.
\end{proof}

\medskip
Now we are ready to prove the main result of this paper, which is Conjecture~B. To simplify the notation, for any graph $G$, we define ${\alpha}_G=\min\{{\alpha}_G(v):v\in V(G)\}$, where ${\alpha}_G(v)$ is defined to be ${\deg_{G}(v)\choose 2}-|E(G_{N(v)})|$. 

Note that ${\alpha}_G=0$ is equivalent to saying that $G$ has a simplicial vertex. For example, let $G$ be the graph which is shown in Figure~\ref{alpha}. Then we have $\alpha_G(1)=\alpha_G(5)=0$, since the vertices $1$ and $5$ are both simplicial vertices. Moreover, $\alpha_G(3)=\alpha_G(4)=1$, and $\alpha_G(2)=2$. Therefore, we have $\alpha_G=0$.

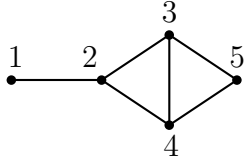
\begin{figure}[hbt]
\begin{center}
\psset{unit=0.6cm}
\begin{pspicture}(-1.5,-0.5)(9.8,2.5)
\psdots(1.5,1)(3.5,1)(5,0)(6.5,1)(5,2)
\psline(1.5,1)(3.5,1)
\psline(3.5,1)(5,0)
\psline(5,2)(3.5,1)
\psline(5,0)(5,2)
\psline(5,0)(6.5,1)
\psline(6.5,1)(5,2)
\uput[120](3.5,1){$2$}
\uput[90](5,2){$3$}
\uput[120](1.85,1){$1$}
\uput[140](5.45,-0.85){$4$}
\uput[90](6.5,1){$5$}
\end{pspicture}
\end{center}
\caption{A graph with $\alpha_G=0$} \label{alpha}
\end{figure}

\begin{Theorem}
For any graph $G\neq P_n$ with $n$ vertices, $\mathrm{reg}(J_G)\leq n-1$.
\end{Theorem}

\begin{proof}
We first prove the theorem for a graph containing a simplicial vertex, or equivalently $\alpha_G=0$. For this, we use induction on the number of the vertices. If $n=2$, then $G$ consists of just two isolated vertices, and hence clearly $J_G=(0)$, and we are done. Let $G$ be a graph on $[n]$, with a simplicial vertex, and assume that $G$ is not a path. We consider two following cases:

Case(i). Suppose that $G$ has a simplicial vertex which is a leaf, say $v$. Then, assume that $w$ is the only neighbor of $v$, and $e=\{v,w\}$ is the edge joining $v$ and $w$. We have $\mathrm{reg}(J_{G\setminus e})=\mathrm{reg}(J_{{(G\setminus e)}_{[n]\setminus v}})$, since $v$ is an isolated vertex of $G\setminus e$. Thus, by \ref{Matsuda-Murai}, $\mathrm{reg}(J_{G\setminus e})\leq n-1$. On the other hand, we have $\reg(J_{G\setminus e}:f_e)=\reg(J_{{(G\setminus e)}_e})$, by \ref{M-colon2}. Note that $v$ is also an isolated vertex of ${(G\setminus e)}_e$, so that we can disregard it in computing the regularity. Thus, $\reg(J_{{(G\setminus e)}_e})\leq n-2$, by the induction hypothesis, since ${(G\setminus e)}_e$ has $w$ as a simplicial vertex. Hence, $\reg(J_{G\setminus e}:f_e)+1\leq n-1$. Thus, by \ref{reg-chordal} (a), we get $\reg(J_G)\leq n-1$.

Case(ii). Suppose that all the simplicial vertices of $G$ have degree greater than one. Let $v$ be a simplicial vertex of $G$ and $v_1,\ldots,v_t$ be all the neighbors of $v$, and $e_1,\ldots,e_t$ be the edges joining $v$ to $v_1,\ldots,v_t$, respectively, where $t\geq 2$. By \ref{reg-chordal} (a) and \ref{colon3}, we have $\reg(J_G)\leq \mathrm{max}\{\reg(J_{G\setminus  e_1}),n-1\}$. Then by applying \ref{reg-chordal} (a) and \ref{colon3} on the graph $G\setminus  e_1$, we get $\reg(J_G)\leq \mathrm{max}\{\reg(J_{G\setminus  \{e_1,e_2\}}),n-1\}$. Since $\mathrm{deg}_{G\setminus \{e_1,\ldots,e_{l}\}}(v)\geq 2$ for $l=1,\ldots,t-2$, we can repeat this process to obtain $\reg(J_G)\leq \mathrm{max}\{\reg(J_{G\setminus \{e_1,\ldots,e_{t-1}\}}),n-1\}$. Note that $G\setminus \{e_1,\ldots,e_{t-1}\}$ is a graph on $n$ vertices in which $v$ is a leaf. Thus, by case~(i), we have $\reg(J_{G\setminus \{e_1,\ldots,e_{t-1}\}}))\leq n-1$. Thus, $\reg(J_G)\leq n-1$.

So, the result follows for the graphs with a simplicial vertex. \\

Now, suppose on the contrary that there exists a graph $G$ on $[n]$ which does not have any simplicial vertex (in particular, $G$ is not a path) and $\mathrm{reg}(J_G)\geq n$. We may assume that $G$ has the least number of vertices, $n$, among the graphs for which the conjecture does not hold. Moreover, we assume that ${\alpha}_G$ is the minimum among the graphs on $n$ vertices with this property. Since $G$ does not contain any simplicial vertex, we have $\alpha_G\geq 1$, and hence there exists a vertex $v$ of $G$ which has two neighbors, say $v_1$ and $v_2$, which are not adjacent in $G$, and ${\alpha}_G={\alpha}_G(v)$. Let $e=\{v_1,v_2\}$. Now, by \ref{reg-chordal} (b), we have
\begin{eqnarray}\label{sequence}
\mathrm{reg}(J_G)\leq \mathrm{max}\{\mathrm{reg}(J_{G\cup e}), \reg(J_{G}:f_e)+2\}.
\end{eqnarray}
We have ${\alpha}_{G\cup e}(v)={\alpha}_G(v)-1$, and hence ${\alpha}_{G\cup e}\leq {\alpha}_{G}-1$. Since $G\cup e$ has $n$ vertices, we have $\mathrm{reg}(J_{G\cup e})\leq n-1$, by our choice of $G$. Note that $G\cup e$, as well as $G$, is not a path.

Now, we show that $\reg(J_{G}:f_e)+2\leq n-1$. By \ref{M-colon2}, we have $J_{G}:f_e=J_{G_e}+I_{G\cup e}$. Since $v_1,v,v_2$ is a path between $v_1$ and $v_2$ in $G$, we have $I_{G\cup e}=(x_v,y_v)+I_{(G\setminus v)\cup e}$, and hence $J_{G}:f_e=J_{G_e}+I_{G\cup e}=J_{{(G\setminus v)}_e}+I_{(G\setminus v)\cup e}+(x_v,y_v)$, by \ref{M-colon2}. Thus, $\reg(J_{G}:f_e)=\reg(J_{{(G\setminus v)}_e}+I_{(G\setminus v)\cup e})$. By \ref{M-colon2}, $\mathrm{reg}(J_{G\setminus v}:f_e)=\reg(J_{{(G\setminus v)}_e}+I_{(G\setminus v)\cup e})$, so that $\reg(J_{G}:f_e)=\mathrm{reg}(J_{G\setminus v}:f_e)$. On the other hand, we have
$\mathrm{reg}(J_{G\setminus v}:f_e)+2\leq \mathrm{max}\{\mathrm{reg}(J_{G\setminus v}), \reg(J_{(G\setminus v)\cup e})+1\}$, by \ref{reg-chordal} (c). By \ref{Matsuda-Murai}, we have $\reg(J_{G\setminus v})\leq n-1$. Now, it suffices to show that $\reg(J_{(G\setminus v)\cup e})\leq n-2$.

First, we claim that $(G\setminus v)\cup e$ is not a path. To prove the claim, suppose on the contrary that $(G\setminus v)\cup e$ is a path over $n-1$ vertices. Then, $G\setminus v$ is the disjoint union of two paths $P_t$ and $P_s$ on two different sets of vertices, where $t+s=n-1$. Note that $e$ joins a vertex of minimum degree of $P_t$ and a vertex of minimum degree of $P_s$, in $(G\setminus v)\cup e$. Moreover, $v$ is adjacent to these two vertices in $G$. So, if $s\leq 2$ or $t\leq 2$, then $G$ has a simplicial vertex which is a contradiction, by our choice of $G$. So, suppose that $t\geq 3$ and $s\geq 3$. Then $v$ is adjacent to both of the leaves of $P_t$ and the leaves of $P_s$ in $G$, since otherwise $G$ has a leaf, and hence a simplicial vertex which contradicts the choice of $G$. Now suppose that $u$ is a neighbor of a leaf of $P_t$, say $w$. If $u$ is adjacent to $v$, then $w$ is a simplicial vertex of $G$ which is a contradiction, because of our choice of $G$. So, suppose that $u$ is not adjacent to $v$. Then $u$ has just two neighbors in $G$ which are not adjacent to each other, and hence $\alpha_G(u)=1$. On the other hand, $\alpha_G(v)\geq 6$, because $v$ is adjacent to at least four vertices, namely the leaves of $P_t$ and $P_s$, and none of those leaves are adjacent to each other in $G$. So, we get a contradiction, since by the definition of $\alpha_G$, we have $\alpha_G=\alpha_G(v)\leq \alpha_G(u)$. Therefore, $(G\setminus v)\cup e$ is not a path and the claim follows.

Thus, by the choice of $G$, we have $\reg(J_{(G\setminus v)\cup e})\leq n-2$, since $(G\setminus v)\cup e$ has $n-1$ vertices. Therefore, by (\ref{sequence}), we get $\reg(J_G)\leq n-1$, which is a contradiction to our assumption, and the desired result follows.
\end{proof}

\section{ Regularity under the join product}\label{join}

\noindent In this section we focus on the join of two graphs and determine the regularity of the binomial edge ideal of the join of two graphs in terms of the original graphs'. Consequently, we gain some results on complete $t$-partite graphs, which generalize some previous known results.

We need the next proposition from \cite{KS} which for completeness we give a proof for it here. If $H$ is a graph with connected components $H_1,\ldots,H_r$, then we denote it by $\bigsqcup_{i=1}^r H_i$.

\begin{Proposition}\label{both disconnected 1}
Suppose that $G_1=\bigsqcup_{i=1}^r G_{1i}$ and $G_2=\bigsqcup_{i=1}^s G_{2i}$ are two graphs on disjoint sets of vertices
$[n_1]=\bigcup_{i=1}^r [n_{1i}]$ and $[n_2]=\bigcup_{i=1}^s [n_{2i}]$,
respectively, where $r,s\geq 2$. Then we have
$$\mathcal{C}(G_1*G_2)=\{\emptyset\}\cup \big{(}(\bigcirc_{i=1}^{r}\mathcal{C}(G_{1i}))\circ \{[n_2]\} \big{)}\cup \big{(}(\bigcirc_{i=1}^{s}\mathcal{C}(G_{2i}))\circ \{[n_1]\} \big{)}.$$
\end{Proposition}

\begin{proof}
Let $G:=G_1*G_2$ and $T\in (\bigcirc_{i=1}^{r}\mathcal{C}(G_{1i}))\circ \{[n_2]\}$. So, $T=[n_2]\cup (\bigcup_{i=1}^r T_{1i})$, where $T_{1i}\in \mathcal{C}(G_{1i})$, for $i=1,\ldots,r$. We show that $T$ has cut point property. Let $j\in T$. If $j\in T_{1i}$, for some $i=1,\ldots,r$, then $G_{([n]\setminus T)\cup \{j\}}={G_{1i}}_{([n_{1i}]\setminus T_{1i})\cup \{j\}}\sqcup (\bigsqcup_{l=1,l\neq i}^r {G_{1l}}_{([n_{1l}]\setminus T_{1l})})$. In this case, $j$ is a cut point of ${G_{1i}}_{([n_{1i}]\setminus T_{1i})\cup \{j\}}$, since $T_{1i}\in \mathcal{C}(G_{1i})$. So that $j$ is also a cut point of $G_{([n]\setminus T)\cup \{j\}}$. If $j\in [n_2]$, then $G_{([n]\setminus T)\cup \{j\}}=j*\bigsqcup_{i=1}^r {G_{1i}}_{([n_{1i}]\setminus T_{1i})}$. So, $j$ is a cut point of $G_{([n]\setminus T)\cup \{j\}}$, since $G_{([n]\setminus T)}$ is disconnected. Thus, in both cases, $T$ has cut point property. If $T\in (\bigcirc_{i=1}^{s}\mathcal{C}(G_{2i}))\circ \{[n_1]\}$, then similarly, we have $T\in \mathcal{C}(G)$.
For the other inclusion, let $\emptyset \neq T\in \mathcal{C}(G)$. If $T$ does not contain $[n_1]$ and $[n_2]$, then $G_{[n]\setminus T}$ is connected, and hence no element $i$ of $T$ is a cut point of $G_{([n]\setminus T)\cup \{i\}}$. So, we have $[n_1]\subseteq T$ or $[n_2]\subseteq T$. Suppose that $[n_1]\subseteq T$. Then, $T=[n_1]\cup (\bigcup_{i=1}^s T_{2i})$, where $T_{2i}\subseteq [n_2]$, for $i=1,\ldots,s$. Let $1\leq i\leq s$. If $T_{2i}=\emptyset$, then, clearly, $T_{2i}\in \mathcal{C}(G_{2i})$. If $T_{2i}\neq \emptyset$, then
each $j\in T_{2i}$, is a cut point of $G_{([n]\setminus T)\cup \{j\}}$, since $T\in \mathcal{C}(G)$. So that $j$ is a cut point of ${G_{2i}}_{([n_{2i}]\setminus T_{2i})\cup \{j\}}$, because $j\in T_{2i}$ and $G_{2i}$'s are on disjoint sets of vertices. Thus, $T_{2i}\in \mathcal{C}(G_{2i})$. Therefore, $T\in (\bigcirc_{i=1}^{s}\mathcal{C}(G_{2i}))\circ \{[n_1]\}$. If $[n_2]\subseteq T$, then similarly we get $T\in (\bigcirc_{i=1}^{r}\mathcal{C}(G_{1i}))\circ \{[n_2]\}$.
\end{proof}

Note that the join of two complete graphs is also obviously complete, so that its binomial edge ideal has a linear resolution, by \cite[Theorem~2.1]{SK}, and hence its regularity is equal to $2$. The following theorem is the main result of this section.

\begin{Theorem}\label{reg-join}
Let $G_1$ and $G_2$ be graphs on $[n_1]$ and $[n_2]$, respectively, not both complete. Then
$$\mathrm{reg}(J_{G_1*G_2})=\mathrm{max}\{\mathrm{reg}(J_{G_1}),\mathrm{reg}(J_{G_2}),3\}.$$
\end{Theorem}

\begin{proof} Let $G:=G_1*G_2$. Note that since $G$ is not a complete graph, $J_G$ does not have a linear resolution, by \cite[Theorem~2.1]{SK}. So that $\mathrm{reg}(J_G)\geq 3$. On the other hand, by \cite[Proposition~8]{SK1}, $\mathrm{reg}(J_G)\geq\mathrm{reg}(J_{G_1})$ and $\mathrm{reg}(J_G)\geq\mathrm{reg}(J_{G_2})$, because $G_1$ and $G_2$ are induced subgraphs of $G$. So, $\mathrm{reg}(J_G)\geq \mathrm{max}\{\mathrm{reg}(J_{G_1}),\mathrm{reg}(J_{G_2}),3\}$. For the other inequality, first, suppose that $G_1$ and $G_2$ are both disconnected graphs. Let $G_1=\bigsqcup_{i=1}^r G_{1i}$ and $G_2=\bigsqcup_{i=1}^s G_{2i}$ be two graphs on disjoint sets of vertices
$[n_1]=\bigcup_{i=1}^r [n_{1i}]$ and $[n_2]=\bigcup_{i=1}^s [n_{2i}]$,
respectively, where $r,s\geq 2$. By \ref{both disconnected 1}, $\mathcal{C}(G)=\{\emptyset\}\cup \big{(}(\bigcirc_{i=1}^{r}\mathcal{C}(G_{1i}))\circ \{[n_2]\} \big{)}\cup \big{(}(\bigcirc_{i=1}^{s}\mathcal{C}(G_{2i}))\circ \{[n_1]\} \big{)}$. So, $J_{G}=Q\cap Q'$, where
\begin{equation}
Q=\bigcap_{\substack{
T\in \mathcal{C}(G) \\
[n_1]\subseteq T
}}
P_T(G)~~,~~Q'=\bigcap_{\substack{
T\in \mathcal{C}(G) \\
[n_1]\nsubseteq T
}}
P_T(G).
\nonumber
\end{equation}
Thus, we have
\begin{equation}
Q=(x_i,y_i:i\in [n_1])+\bigcap_{\substack{
T\in \mathcal{C}(G) \\
[n_1]\subseteq T
}}
P_{T\setminus [n_1]}(G_2)
\nonumber
\end{equation}
and
\begin{equation}
Q'=P_{\emptyset}(G)\cap\big{(}\bigcap_{\substack{
\emptyset\neq T\in \mathcal{C}(G) \\
[n_1]\nsubseteq T
}}
P_T(G)\big{)}
=P_{\emptyset}(G)\cap \big{(}(x_i,y_i:i\in [n_2])+\bigcap_{\substack{
T\in \mathcal{C}(G) \\
[n_2]\subseteq T
}}
P_{T\setminus [n_2]}(G_1)\big{)}.
\nonumber
\end{equation}
So, one can see that $Q=(x_i,y_i:i\in [n_1])+J_{G_2}$, $Q'=J_{K_n}\cap \big{(}(x_i,y_i:i\in [n_2])+J_{G_1}\big{)}$ and $Q+Q'=(x_i,y_i:i\in [n_1])+J_{K_{n_2}}$.
Now, consider the short exact sequence $$0\rightarrow J_G\rightarrow Q\oplus Q'\rightarrow Q+Q'\rightarrow 0.$$
By \cite[Corollary~18.7]{P}, we have $\mathrm{reg}(J_G)\leq \mathrm{max}\{\mathrm{reg}(Q),\mathrm{reg}(Q'),\mathrm{reg}(Q+Q')+1\}$. On the other hand, we have $\mathrm{reg}(Q)=\mathrm{reg}(J_{G_2})$, $\mathrm{reg}(Q')\leq \mathrm{max}\{\mathrm{reg}(J_{G_1}),\mathrm{reg}(K_{n_1})+1=3\}$ (by using a suitable short exact sequence as above), and $\mathrm{reg}(Q+Q')=\mathrm{reg}(K_{n_2})+1=3$. Hence, $\mathrm{reg}(J_G)\leq \mathrm{max}\{\mathrm{reg}(J_{G_2}),\mathrm{reg}(J_{G_1}),3\}$. Now, suppose that $G_1$ or $G_2$ is connected. We add an isolated vertex $v$ to $G_1$ and an isolated vertex $w$ to $G_2$. Thus, we obtain two disconnected graphs $G_1'$ and $G_2'$. So, by the above discussion, we have $\mathrm{reg}(J_{G_1'*G_2'})\leq \mathrm{max}\{\mathrm{reg}(J_{G_1'}),\mathrm{reg}(J_{G_2'}),3\}$. But, clearly, we have $\mathrm{reg}(J_{G_1'})=\mathrm{reg}(J_{G_1})$ and $\mathrm{reg}(J_{G_2'})=\mathrm{reg}(J_{G_2})$, so that $\mathrm{reg}(J_{G_1'*G_2'})\leq \mathrm{max}\{\mathrm{reg}(J_{G_1}),\mathrm{reg}(J_{G_2}),3\}$. Thus, the result follows by \cite[Proposition~8]{SK1}, since $G_1*G_2$ is an induced subgraph of $G_1'*G_2'$.
\end{proof}

The following corollary generalizes the result of \cite{SZ} on the regularity of complete bipartite graphs.

\begin{Corollary}\label{reg-t-partite}
Let $G$ be a complete $t$-partite graph, where $t\geq 2$. If $G$ is not complete, then $\mathrm{reg}(J_{G})=3$.
\end{Corollary}

\begin{proof}
We use induction on $t\geq 2$, the number of parts. If $t=2$, then $G$ is the join of two graphs each consisting of some isolated vertices. So, the regularity of the binomial edge ideal of each of them is $0$. Thus, by \ref{reg-join}, we have $\mathrm{reg}(J_{G})=3$. Now, suppose that $t>2$ and the result is true for every complete $(t-1)$-partite graph which is not complete. Let $V_1,\ldots,V_t$ be the partition of the vertices of $G$ to $t$ parts. Hence, we have
$G=G_{V_t}*G_{V\setminus V_t}$, and $G_{V\setminus V_t}$ is a complete $(t-1)$-partite graph. If $G_{V\setminus V_t}$ is a complete graph, then $|V_t|>1$, since, otherwise, $G$ is a complete graph, a contradiction. So, by \ref{reg-join}, $\mathrm{reg}(J_{G})=3$. If $G_{V\setminus V_t}$ is not complete, then by the induction hypothesis, we have $\mathrm{reg}(J_{G_{V\setminus V_t}})=3$. Thus, again by \ref{reg-join}, the result follows.
\end{proof}

\medskip
We end this section with some remarks regarding Conjecture~A. \\

Note that by \ref{reg-join}, we have if $G$ is a (multi)-fan graph (i.e. $K_1*\bigsqcup_{i=1}^t P_{n_i}$, for some $t\geq 1$, which might be a non-closed graph), then $\mathrm{reg}(J_G)=c(G)+1$. This implies that if Conjecture~A is true, then the given bound is sharp.

\begin{Corollary}\label{ConjB-join}
Let $G_1$ and $G_2$ be two graphs on $[n_1]$ and $[n_2]$, respectively. If Conjecture~A is true for $G_1$ and $G_2$, then it is also true for $G_1*G_2$.
\end{Corollary}

\begin{proof}
By \ref{reg-join}, it is enough to note that $c(G_1*G_2)=c(G_1)c(G_2)$, and that if $G_1$ and $G_2$ are complete graphs, then $G_1*G_2$ is also complete and Conjecture~A is true for it.
\end{proof}

\textbf{Acknowledgments:} The authors would like to thank Professor J\"{u}rgen Herzog for his useful and valuable comments.


\begin{thebibliography}{}


\bibitem{D} A. Dokuyucu, {\em Extremal Betti numbers of some classes of binomial edge ideals}, (2013), arXiv:1310.2903.


\bibitem{EHH} V. Ene, J. Herzog and T. Hibi, {\em Cohen-Macaulay binomial edge ideals}, Nagoya Math. J. 204 (2011), 57-68.


\bibitem{EZ} V. Ene and A. Zarojanu, {\em On the regularity of binomial edge ideals}, Math. Nachr. 288,  No. 1 (2015), 19-24. 


\bibitem{HHHKR} J. Herzog, T. Hibi, F. Hreinsdotir, T. Kahle and J. Rauh, {\em Binomial edge ideals and conditional independence statements},
Adv. Appl. Math. 45 (2010), 317-333.


\bibitem{KS} D. Kiani and S. Saeedi Madani, {\em Binomial edge ideals with pure resolutions}, Collect. Math. 65 (2014), 331-340.


\bibitem{KS1} D. Kiani and S. Saeedi Madani, {\em Some Cohen-Macaulay and unmixed binomial edge ideals}, to appear in Comm. Algebra.


\bibitem{MM} K. Matsuda and S. Murai, {\em Regularity bounds for binomial edge ideals}, Journal of Commutative Algebra. 5(1) (2013), 141-149.


\bibitem{MSh} F. Mohammadi and L. Sharifan, {\em Hilbert function of binomial edge ideals}, Comm. Algebra 42 (2014), 688-703.


\bibitem{O} M. Ohtani, {\em Graphs and ideals generated by some 2-minors}, Comm. Algebra. 39 (2011), 905-917.


\bibitem{P} I. Peeva, {\em Graded syzygies}, Springer, (2010).


\bibitem{SK} S. Saeedi Madani and D. Kiani, {\em Binomial edge ideals of graphs}, The Electronic Journal of Combinatorics. 19(2) (2012), $\sharp$ P44.


\bibitem{SK1} S. Saeedi Madani and D. Kiani, {\em On the binomial edge ideal of a pair of graphs}, The Electronic Journal of Combinatorics. 20(1) (2013), $\sharp$ P48.

\bibitem{SZ} P. Schenzel and S. Zafar, {\em Algebraic properties of the binomial edge ideal of a complete bipartite graph}, to appear in An. St. Univ. Ovidius Constanta, Ser. Mat.


\bibitem{ZZ} Z. Zahid and S. Zafar, {\em On the Betti numbers of some classes of binomial edge ideals}, The Electronic Journal of Combinatorics. 20(4) (2013), $\sharp$ P37.

\end{thebibliography}
\end{document}